\documentclass[11pt]{amsart}

\usepackage{amsthm,amsmath,amsfonts, amssymb}

\def \RR{\mathbb R}

\def \SS{\mathbb S}

\def \A{\mathcal A}

\def \C{\mathcal C}

\def \F{\mathcal F}

\def \O{\mathcal O}

\def \indic#1{1\!\!1_{\!#1}}
\def \W{\widetilde{W}}

\def \close#1{\overline{#1}}

\theoremstyle{plain}
\newtheorem{theorem}{Theorem}[section]
\newtheorem{proposition}[theorem]{Proposition}
\newtheorem{corollary}[theorem]{Corollary}
\newtheorem{lemma}[theorem]{Lemma}

\theoremstyle{definition}
\newtheorem{remark}[theorem]{Remark}

\begin{document}

\title[Brownian motion in a cone]{Brownian motion conditioned to stay in a cone}
\author[R.~Garbit]{Rodolphe Garbit}
\date{May 21, 2009}
\address{Laboratoire de Math\'ematiques Jean Leray, UMR CNRS 6629\\
Universit\'e de Nantes\\ BP 92208\\ 44322 Nantes Cedex 3\\ France.}
\email{rodolphe.garbit@univ-nantes.fr}

\begin{abstract}
A result of R.~Durrett, D.~Iglehart and D.~Miller states that Brownian meander is Brownian motion conditioned to stay positive for a unit of time, in the sense that it is the weak limit, as $x$ goes to $0$, of Brownian motion started at $x>0$ and  conditioned to stay positive for a unit of time.
We extend this limit theorem to the case of multi\-dimensional Brownian motion conditioned to stay in a smooth convex cone. 
\end{abstract}

\maketitle

\section{Introduction}\label{intro}
The purpose of this paper is to prove the existence of a process which is, in some sense, a multidimensional Brownian motion started at the vertex of a  smooth convex cone and conditioned to stay in it for a unit of time.

Let $\C_{\infty}$ be the space of continuous functions $w:[0,+\infty)\to\RR^d$, $d\geq 1$, endowed with the topology of uniform convergence on compact subsets,
and let $\F$ be the corresponding Borel $\sigma$-algebra. We shall use $\C_{\infty}$ as a concise notation for $(\C_{\infty},\F)$.
Weak convergence in the space of probability measures on $\C_{\infty}$ will be denoted by the symbol $\Rightarrow$.

Let $\{X_t,t\geq 0\}$ be the canonical process on $\C_{\infty}$ for which $X_t(w)=w(t)$ for any $w\in\C_{\infty}$.
Consider an open cone $C$ with vertex at the origin $0$ and let $\tau_C=\inf\{t>0:X_t\notin C\}$ be the first exit time of the canonical process from~$C$.
For any $x\in C$ we define the law $\W^C_{x,1}$ of the Brownian motion started at $x$ and conditioned to stay in $C$ for a unit of time by
the formula
$$\W^C_{x,1}(*)=W_x(*\,\vert\,\tau_C>1)\;,$$
where $W_x$ is the distribution on $\C_{\infty}$ of the standard $d$-dimensional Brownian motion started at $x$.

The main result of this paper is the following theorem (the precise definition of a nice cone is given in Section~\ref{ccones}; for example, any circular or ellipsoidal cone is nice).
\begin{theorem}\label{mainthm}
Suppose $C$ is a nice cone.
As $x\in C$ goes to $0$, the law $\W^C_{x,1}$ converges weakly on $\C_{\infty}$ to a limit $\W^C_{0,1}$.\\
For any $t\in (0,1]$, the entrance law $\W^C_{0,1}(X_t\in dy)$ has the density $e(t,y)$ (w.r.t. Lebesgue measure) given by
the formula~\eqref{densite}.
\end{theorem}

Theorem~\ref{mainthm} is  the multidimensional analog of Durrett, Iglehart and Miller result~(\cite{DIM77}, Theorem 2.1) in which they consider Brownian motion conditioned to stay positive for a unit of time and identify the limit as the Brownian meander. In the case of two-dimensional Brownian motion, Theorem~\ref{mainthm} is due to Shimura~(\cite{Shi85}, Theorem 2). 

For geometric reasons, the extension of Shimura's result to higher dimensions is not straightforward.
We first prove the convergence of the finite-dimensional distribution with the help of an explicit formula for the heat kernel of a cone given by Ba\~nuelos and Smits in~\cite{BS97}.
Then, we prove tightness of the laws $\W^C_{x,1}$ as $x\in C\to 0$ using a principle already present in Shimura's article: If $\W^C_{x,1}$ converges weakly as $x$ tends to any point $x_0\in\partial C\setminus\{0\}$, then the tightness as $x\in C\to 0$ follows. 
For a two-dimensional cone, proving weak convergence of $\W^C_{x,1}$ as $x\to x_0\in\partial C\setminus\{0\}$ is quite easy because $\partial C$ is locally linear at $x_0$, so the proof is nearly the same as in the one-dimensional case.
But in higher dimensions the geometry of the boundary of a cone is not so simple and we are led to a quite more general problem: Given an open set $U$ and a point $x_0\in\partial U$, does the law $\W^U_{x,1}$ of Brownian motion started at $x\in U$ and conditioned to stay in $U$ for unit of time converge weakly as $x\in U\to x_0$? The major part of this paper is in fact concerned with the study of this question.

In Section~\ref{gpotcl}, we consider the general problem of Brownian motion conditioned to stay in an open set $U$ and give some useful properties of the conditioned laws $\W^U_{x,1}$, such as the Markov property and a form of continuity with respect to the variable $x$. In Section~\ref{thsc}, we recall Durret-Iglehart-Miller result on Brownian motion conditioned to stay positive which extends immediately to the case of Brownian motion conditioned to stay in a half-space. From the half-space case, we then derive in Section~\ref{precond} a  convergence theorem for  $\W^U_{x,1}$ as $x\to x_0\in\partial U$ when $U$ is nice at $x_0$.
This new result is based on the \textit{ball estimate} (Lemma~\ref{CONTRPOURBOULTPSSORT}) which constitute the heart of this paper.
Finally, in Section~\ref{ccones} we present a complete proof of Theorem~\ref{mainthm}, and we give some properties of the limit process, such as the distribution of its first exit time from the cone after time $1$.

\subsubsection*{Notations}
If $\mu$ is a probability measure on a space $(X,\A)$, we will denote by $\mu(f)$ the expectation of a measurable function with respect to $\mu$.
For a set $A\in\A$ and a measurable function $f$, the notation $\mu(A;f)$ stands for $\mu(\indic{A}\times f)$, where $\indic{A}$ is the characteristic function of the set $A$. For consistency, $\mu(A;B)$ will often be preferred to
$\mu(A\cap B)$.

\section{Basic facts about the conditioned laws}\label{gpotcl}

\subsection{Markov property}
Let $U$ be an open subset of $\RR^d$ and let $\tau_U$ be the first exit time from $U$. For any $x\in U$ and $t>0$, we set
$$\W^U_{x,t}(*)=W_x(*\,\vert\,\tau_U>t)=\frac{W_x(*;\tau_U>t)}{W_x(\tau_U>t)}\;.$$
For convenience, we will also use the notation
$\W^U_{x,t}:=W_x$ for any $t\leq 0$ and $x\in\RR^d$.

Let $\F_t$ be the $\sigma$-algebra generated by the random variables $\{X_s,s\leq t\}$.
The shift operator $\theta_t$ on $\C_{\infty}$ is defined by $\theta_t(w)(s)=w(t+s)$.
For any optional time $\tau$, we set $\F_{\tau^+}=\{A\in \F:\forall t\geq 0, A\cap\{\tau<t\}\in\F_t\}$. 
The laws $\W^U$ inherit a strong Markov property from Brownian motion. The proof is standard and will be omitted here.

\begin{proposition}[Strong Markov property]\label{MarkovForteMBC}Let $x\in U$ and $t>0$. For any optional time $\tau$, any $A\in\F_{\tau^+}$ and any positive measurable function $f(s,w):[0,+\infty)\times \C_{\infty}\to\RR$, we have
$$\W^U_{x,t}\left(A;\tau < t;f(\tau,\theta_{\tau}\cdot)\right)=\W^U_{x,t}\left(A;\tau< t;\W^U_{X(\tau),t-\tau}(f(s,\cdot))_{|s=\tau}\right)\;.$$
\end{proposition}
\subsection{Continuity}

Let $U$ be an open subset of $\RR^d$. We will say that $U$ is co-regular if $W_x(\tau_{\close{U}}>0)=0$ for every $x\in\partial U$; that is,
a Brownian motion started at any point of the boundary of $U$ visits instantaneously the complement of $\close{U}$.
For such a set,  $\tau_U$ and $\tau_{\close{U}}$ are almost surely equal.
 
\begin{proposition}\label{contthm}Suppose $U$ is co-regular.
Then, for every bounded continuous function $f$ on $\C_{\infty}$, the mapping $(x,t)\mapsto\W^U_{x,t}(f)$ is continuous on $U\times(0,+\infty)$.
\end{proposition}
\begin{proof}Since $W_x(\tau_U>t)>0$ for any $(x,t)\in U\times(0,+\infty)$, it suffices to prove that the mapping $(x,t)\mapsto W_x(f;\tau_U>t)$ is continuous on $U\times(0,+\infty)$.
Suppose $x_n\to x\in U$ and $t_n\to t>0$. Set $\phi(w)=f(w)\,\indic{\{\tau_U>t\}}(w)$
and
$\phi_n(w)=f(x_n-x+w)\,\indic{\{\tau_U>t_n\}}(x_n-x+w)$.
Since $W_{x_n}(f;\tau_U>t_n)=W_x(\phi_n)$ and $W_{x}(f;\tau_U>t)=W_x(\phi)$, it is enough to show that $\phi_n(w)\to\phi(w)$ for $W_x$-almost every $w$. 

Set $\Omega=\{X_0=x;\tau_U=\tau_{\close{U}}\neq t\}$.
It is well-known that $W_x(\tau_U=t)=0$ (\cite{Por78}, Theorem 4.7), and since $U$ is co-regular we also have $W_x(\tau_U=\tau_{\close{U}})=0$.
Hence $W_x(\Omega)=1$. 
Now choose a path $w\in\Omega$.
Using the continuity of $w$, it is easily seen that 
$\indic{\{\tau_U>t_n\}}(x_n-x+w)\to\indic{\{\tau_U>t\}}(w)$.
Hence $\phi_n(w)\to\phi(w)$ for $W_x$-almost every $w$.
\end{proof}

\subsection{Finite-dimensional distributions}
Let $U\subset\RR^d$ be a co-regular open set and $x_0$ a boundary point of $U$. 
We shall now give a sufficient condition for the weak convergence of the finite-dimensional distributions of $\W^U_{x,1}$, as $x$ tends to $x_0$, that only involves the first transitions.

\begin{proposition}\label{critconvmarg} Suppose that for any $t\in(0,1)$, the first transition law $\W^U_{x,1}(X_t\in dy)$ converges weakly as $x\to x_0$ to a probability measure for which $\partial U$ is a null set. Then, the finite-dimensional distributions of $\W^U_{x,1}$ converge weakly as $x\to x_0$ to some probability measures.
\end{proposition}

Proposition~\ref{critconvmarg} follows from the Markov property and the continuity (Propositions~\ref{MarkovForteMBC} and~\ref{contthm}) by standard arguments, so the proof will be omitted here.

For $x\in U$ and $t>0$, the Markov property of Brownian motion gives 
$$
\W^U_{x,1}(X_t\in dy)=\frac{W_x(X_t\in dy;\tau_U>t;W_{X_t}(\tau_U>1-t))}{W_x(\tau_U>1)}\;.
$$
Since the transitions $W_x(X_t\in dy;\tau_U>t)$ of Brownian motion killed on the boundary of $U$ have densities $p^U(t,x,y)$ with respect to Lebesgue measure $dy$, we get
$$
\W^U_{x,1}(X_t\in dy)=\frac{p^U(t,x,y)}{W_x(\tau_U>1)}W_y(\tau_U>1-t)\,dy\;.
$$
Hence, proving convergence of the finite-dimensional distributions of $\W^U_{x,1}$ consists essentially in finding an asymptotic formula for the heat kernel $p^U(t,x,y)$ as $x\to x_0$.

\subsection{Neat convergence}

Let $U\subset \RR^d$ be a co-regular open set and let $x_0$ be a boundary point of~$U$.
Suppose there exists a law $\W^U_{x_0,1}$ on $\C_{\infty}$ such that $\W^U_{x,1}\Rightarrow \W^U_{x_0,1}$ as $x\in U$ tends to $x_0$. We will say that the convergence is neat (or that $\W^U_{x,1}$ converges neatly to $\W^U_{x_0,1}$ as $x\in U\to x_0$) if the limit process does not leave $U$ before time~$1$, \textit{i.e.} $\W^U_{x_0,1}(\tau_U>1)=1$. The next proposition gives a sufficient condition of neat convergence and states that  the Markov property then holds for the limit process.

\begin{proposition}\label{prolongement}
Suppose $\W^U_{x,1}\Rightarrow\W^U_{x_0,1}$ as $x\in U\to x_0$.\\
If $\W^U_{x_0,1}(X_t\in\partial U)= 0$ for all $t\in(0,1)$, then the convergence is neat and the limit process $\W^U_{x_0,1}$ satisfies the following Markov property:\\
For all $t> 0$, $A\in \F_{t^+}$ and $B\in\F$,
\begin{equation}\label{markovprolongement}
\W^U_{x_0,1}(A;\theta_t^{-1}B)=\W^U_{x_0,1}\left(A;\W^U_{X_t,1-t}(B)\right)\;.
\end{equation}
\end{proposition}
\begin{proof}
Once it has been observed that the assumptions ensure that $\W^U_{x_0,1}(X_t\in U)=1$ for all $t\in(0,1)$, 
the Markov property of $\W_{x_0,1}^U$ follows from Propositions~\ref{MarkovForteMBC} and~\ref{contthm} by standard arguments (see for example~\cite{Por78}, Proof of Theorem 3.2). 
Then it remains to prove that $\W^U_{x_0,1}(\tau_U>1)=1$. But by the Markov property of $\W^U_{x_0,1}$, we have
$$\W^U_{x_0,1}(\forall t\in(\epsilon,1], X_t\in U)=\W^U_{x_0,1}\left(\W^U_{X_{\epsilon,1-\epsilon}}(\tau_U>1-\epsilon)\right)=1$$
for all $\epsilon\in (0,1)$, thus
the expected result follows by letting $\epsilon\to 0$.
\end{proof}

\begin{remark}\label{01Law} Letting $t\to 0$ in \eqref{markovprolongement} would give a zero-one law for $\W^U_{x_0,1}$ (\textit{i.e.} $\W^U_{x_0,1}(A)=0$ or $1$ if $A\in\F_{0^+}$) if we had the stronger assumption that $\W^U_{x,t}\Rightarrow\W^U_{x_0,1}$ as $(x,t)\to (x_0,1)$. Note that in the special case where $U=C$ is a cone, the last convergence follows from the hypothesis $\W^C_{x,1}\Rightarrow\W^C_{x_0,1}$ because of the scaling property of Brownian motion. More precisely, let $K_t$ be the scaling operator defined for all $w\in \C_{\infty}$ by
$K_t(w)(s)=\sqrt{t}w(s/t)$.
Recall that $W_0$ is $K_t$-invariant. From the scaling invariance of the cone $C$, it is easily checked that
$\W^C_{x,t}=\W^C_{x/\sqrt{t},1}\circ K_t^{-1}$.
If $(x,t)\to(x_0,1)$, then $x/\sqrt{t}\to x_0$ and by the continuous mapping theorem (\cite{Bil99}, Theorem 2.7) we get
$$\W^C_{x,t}\Rightarrow\W^C_{x_0,1}\circ K_1^{-1}=\W^C_{x_0,1}\;.$$
Therefore, the zero-one law follows under the hypothesis of Proposition~\ref{prolongement}.
\end{remark}

\section{The half-space case}\label{thsc}
\subsection{Brownian motion conditioned to stay positive}
We will now recall the one dimensional theorem of Durret, Iglehart and Miller~(\cite{DIM77}, Theorem~2.1) and give a sketch of their proof.
Auxiliary results such as Lemma~\ref{unconv} and~\ref{uncond} shall also be used in Section~\ref{precond}.
Throughout this section we set $d=1$ and we denote by $\tau_+=\inf\{t>0:X_t\leq 0\}$ the first exit time from the half-line $(0,+\infty)$. The related conditional laws will be denoted by $\W^+_{x,1}$.

The Brownian meander is an inhomogeneous Markov process with continuous path that is obtained from Brownian motion by the following path transformation: Let $\sigma=\max\{t< 1:X_t=0\}$ be the time of the last zero before time $1$, and let
$$\widetilde{X}_t=\frac{1}{\sqrt{1-\sigma}}\left\vert X(\sigma+t(1-\sigma))\right\vert\;.$$
Then, with respect to Wiener measure $W_0$, the process $(\widetilde{X}_t)_{t\geq 0}$ is the Brownian meander.
Let $\W^+_{0,1}$ be the law of the Brownian meander on $\C_{\infty}$. We have the following theorem.
\begin{theorem}[\cite{DIM77}, Theorem 2.1]\label{dimone}
As $x>0$ tends to $0$, 
$\W^+_{x,1}\Rightarrow\W^+_{0,1}$.
\end{theorem}
The idea of the proof of Theorem~\ref{dimone} is to turn the  conditioned laws into unconditioned ones by the mean of well-chosen sections of the original process. Let us give some details.
For all $x\geq 0$, introduce the random time 
$$T_x=\inf\{t\geq 0:X_t=x\mbox{ and }X_s>0\mbox{ for all }s\in(t,t+1]\}\;.$$
which is $W_0$-almost surely finite. 
The next lemma is straightforward:
\begin{lemma}[\cite{DIM77}, Lemma 2.2]\label{unconv}
As $x\to 0$, $T_x$ converges almost surely to $T_0$ with respect to $W_0$. 
\end{lemma}

To each time $T_x$ we associate the shift operator $\phi_x:=\theta_{T_x}$ acting on $\C_{\infty}$. We then have:
\begin{lemma}[\cite{DIM77}, Lemma 2.3]\label{uncond}For every $x>0$,
$\W^+_{x,1}=W_0\circ\phi_x^{-1}$.
\end{lemma}

Lemma~\ref{uncond} gives an ``unconditioned'' representation of the laws $\W^+_{x,1}$, $x>0$. It is noteworthy that $W_0\circ\phi_x^{-1}$ also make sense for $x=0$ whereas the definition of $\W^+_{x,1}$ does not. From Lemmas~\ref{uncond} and~\ref{unconv}, it follows by the dominated convergence theorem that 
$$\W^+_{x,1}=W_0\circ\phi_x^{-1}\Rightarrow W_0\circ\phi_0^{-1}$$
as $x\to 0$. Note that the limit law clearly satisfies $W_0\circ\phi_0^{-1}(\tau_+>1)=1$; hence the convergence is neat.
In order to  prove Theorem~\ref{dimone}, it remains to identify the limit with the Brownian meander. This can be done by computing the limit of the finite-dimensional distributions
of the laws $\W^+_{x,1}$ which are easily derived from a classical formula for the joint distribution of Brownian motion and its minimum.
We do not give further detail since no expression of these finite-dimensional distributions will be needed in what follows.

\subsection{Brownian motion conditioned to stay in a half-space}
Theorem~\ref{dimone} can  easily be extended to multidimensional Brownian motion conditioned to stay in a half-space.
Let $d\geq 2$. Because of invariance properties of $d$-dimensional Brownian motion we need only to study the case of the half-space $D=\{x\in\RR^d:x_1>0\}$.
Let $BM$ be a Brownian meander and $B_2,\ldots,B_d$ be one-dimensional Brownian motions such that $BM,B_2,\ldots,B_d$ are mutually independent. The $d$-dimensional process $(BM,B_2,\ldots,B_d)$ will be called $D$-Brownian meander and its law will be denoted by $\W^D_{0,1}$.

\begin{corollary}\label{CONQUDEMIESP} 
As $x\in D\to 0$, $\W^D_{x,1}\Rightarrow \W^D_{0,1}$.
\end{corollary}
\begin{proof}
A Brownian motion conditioned to stay in the half-space $D$ is a Brownian motion whose first coordinate is conditioned to stay positive. Since the coordinates are independent one-dimensional Brownian motions, the result follows immediately from Theorem~\ref{dimone}.
\end{proof}

\begin{remark}It is clear from the definition of the $D$-Brownian meander that it satisfies $\W^D_{0,1}(\tau_D>1)=1$; thus the convergence in Theorem~\ref{CONQUDEMIESP} is neat and $\W^D_{0,1}$ has the Markov property of Proposition~\ref{prolongement}. Moreover, since $D$ is a cone, it follows from Remark~\ref{01Law} that we also have a zero-one law with respect to $\W^D_{0,1}$.
\end{remark}

\section{Preconditioning}\label{precond}
We shall now use the results of Section~\ref{thsc} in order to obtain a convergence theorem  for the Brownian motion conditioned to stay in a set satisfying some regularity and convexity assumptions (Theorem~\ref{convnice}).
Section~\ref{clftcb} introduces the idea of preconditioning and explains how it can be applied to  the convergence problem.
The proposed method requires an estimate that is studied in Section~\ref{ballestimate}.
This finally leads us to introduce the class of \textit{nice sets} for which we solve the convergence problem in Section~\ref{nicenice}.

\subsection{Changing laws for the convergence problem}\label{clftcb}
Let $U\subset\RR^d$ be a co-regular open set with $0\in\partial U$. Recall that the definition of $\W^U_{x,1}$ by the formula
$$\W^U_{x,1}(*)=\frac{W_x(*;\tau_U>1)}{W_x(\tau_U>1)}\;,$$
does not make any sense for $x=0$ since $W_0(\tau_U>1)=0$.

Now suppose  $U$ is a subset of the half-space $D$. Then a Brownian motion conditioned to stay in $U$ is also a Brownian motion conditioned to stay in $D$ and then conditioned to stay in $U$, that is:
\begin{equation}
\W^U_{x,1}(*)=\W^D_{x,1}(*\,\vert\,\tau_U>1)=\frac{\W^D_{x,1}(*;\tau_U>1)}{\W^D_{x,1}(\tau_U>1)}\;.
\end{equation}
This simple identity is what we call \textit{preconditioning}, for if we take it as a definition, it is the same as before except we have changed the initial law of the paths ($W_x\leftrightarrow\W^D_{x,1}$) which are now preconditioned to stay in $D$. The gain is  that, although $W_0(\tau_U>1)=0$, we might have $\W^D_{0,1}(\tau_U>1)>0$ if the boundary of $U$ is smooth enough at $0$. If $\W^D_{0,1}(\tau_U>1)>0$, we will 
set 
\begin{equation}\label{predef}
\W^U_{0,1}(*):=\W_{0,1}^D(*\,\vert\,\tau_U>1)\;.
\end{equation}

It follows easily from the Markov property of $\W^D_{0,1}$ that the condition $\W^D_{0,1}(\tau_U>1)>0$ is satisfied if $\W^D_{0,1}(\tau_U>0)=1$, that is if $0$ is $\W^D_{0,1}$-irregular for $U^c$.
In~\cite{Bur86}, Corollary 3.1, Burdzy gives an irregularity criterion relative to the $D$-Brownian meander. 
The next lemma is an easy consequence of his result:
\begin{lemma}\label{ballpos}
If $B$ is a ball with radius $r>0$ and center at $(r,0,\ldots,0)$, then
$\W^D_{0,1}(\tau_{B}>0)=1$.
\end{lemma}
In particular, if there exists an open ball $B$ tangent to $\partial D$ at $0$ and such that $B\subset U\subset D$, then 
$\W^D_{0,1}(\tau_U>0)=1$. Thus the law $\W^U_{0,1}$ can be defined by relation~\eqref{predef}.
We point out the fact that applying Burdzy criterion to
a proper cone $C\subset D$ with vertex at the origin gives $\W^D_{0,1}(\tau_C>0)=0$; hence a law $\W^C_{0,1}$ can not be defined directly.

Suppose $U$ is such that $\W^D_{0,1}(\tau_U>1)>0$.
The question we then have to answer is the following: Does the convergence $\W^D_{x,1}\Rightarrow\W^D_{0,1}$ imply that
$$\W^D_{x,1}(*\,\vert\,\tau_U>1)\Rightarrow \W^D_{0,1}(*\,\vert\,\tau_U>1)\;?$$
Unless $U$ is locally linear at $0$, this can not follow directly from the continuous mapping theorem, since $\tau_U$ is $\W^D_{0,1}$-almost surely discontinuous. To overcome this problem we will use an estimate that we present in next section.

\subsection{The ball estimate}\label{ballestimate}

Fix $d\geq 2$. 
We shall note $X_1(t),\ldots,X_d(t)$, the coordinates of the canonical process $X(t)$.
Let $D$ be the half-space $\{x\in\RR^d:x_1>0\}$ and $B$ the open ball with center at $e_1=(1,0,\ldots,0)$ and radius $1$. 
 Let $E$ be the set of all $(d-1)$-uples 
 $(\epsilon_2,\ldots,\epsilon_d)$ with $\epsilon_i=\pm 1$. For all $\epsilon=(\epsilon_2,\ldots,\epsilon_d)\in E$, let  $\overline{\epsilon}=(-\epsilon_2,\ldots,-\epsilon_d)$ be the opposite of $\epsilon$. 
 We define a familly of $2^{d-1}$ disjoint subsets of $D$ indexed by $E$ by setting 
$D_{\epsilon}=\{x\in D:\epsilon_2 x_2,\ldots,\epsilon_d x_d> 0\}$.
Let $H$ be the hyperplane $\{x_1=1\}$ and, for all $x\in\RR^d$, let $B(x)$ be the open ball with center at $x$ and radius $1$.
The next lemma is straightforward:
\begin{lemma}\label{UNPEUDEGEOPOUTOULEMONDE}
If $x\in D_{\epsilon}\cap H$, then $B(x)\cap B^c\cap D_{\overline{\epsilon}}=\emptyset$.
\end{lemma}
We now come to the estimate which is the heart of this section:
\begin{lemma}\label{CONTRPOURBOULTPSSORT} 
$\lim_{s\to 0}\limsup_{\lambda\to 0}\W^{D}_{\lambda e_1,1}(\tau_{B}\leq s)=0$.
\end{lemma}
\begin{proof} 
We will show that
\begin{equation}\label{CONPROSORAVAS}
\limsup_{\lambda\to 0}\W^{D}_{\lambda e_1,1}(\tau_{B}\leq s)\leq 2^{d-1}\,\W^D_{0,1}(\tau_{B}\leq s)
\end{equation}
for all $s>0$, and the lemma will then follow by letting $s\to 0$ since $\W^D_{0,1}(\tau_B=0)=0$ (Lemma~\ref{ballpos}).
For $\lambda\geq 0$, set 
$$T_{\lambda}=\inf\{t\geq 0:X_1(t)=\lambda\hbox{ and } X_1(s)>0,\forall s\in (t,t+1]\}$$
and consider the process  $Z_{\lambda}$ defined by
$$\forall t\geq 0,\quad Z_{\lambda}(t)=X(T_{\lambda}+t)-X(T_{\lambda})+\lambda e_1\;.$$
By independence of the coordinates $X_1,\ldots, X_d$, and Lemma~\ref{uncond}, the process $Z_{\lambda}$ has the distribution $\W^D_{\lambda e_1,1}$ with respect to $W_0$.
Write
\begin{eqnarray}\label{DECOMPPURMIEREVZZ}
W_0(\tau_{B}(Z_{\lambda})\leq s) &\leq & W_0(\tau_{B}(Z_ 0)\leq s+T_{\lambda}-T_0)\\
& &+W_0(\tau_{B}(Z_0)> s+T_{\lambda}-T_0;\tau_{B}(Z_{\lambda})\leq s)\;.\nonumber
\end{eqnarray}
For convenience, we set $u=\tau_{B}(Z_{\lambda})$.
If 
$\tau_{B}(Z_0)>s+T_{\lambda}-T_0$ and $u\leq s$, then $Z_0(u+T_{\lambda}-T_0)=X(T_{\lambda}+u)-X(T_0)$ belongs to $B$;
this means that
$Z_{\lambda}(u)$ belongs to $B(Y_{\lambda})$, where we have put 
\begin{eqnarray*}
Y_{\lambda} & = & X(T_0)-X(T_{\lambda})+(1+\lambda)e_1\\
& = &(1,X_2(T_0)-X_2(T_{\lambda}),\ldots,X_d(T_0)-X_d(T_{\lambda}))\;.
\end{eqnarray*}
Note that $Y_{\lambda}\in H$.
Since $Z_{\lambda}(u)\not\in B$, we see by Lemma~\ref{UNPEUDEGEOPOUTOULEMONDE} that $Z_{\lambda}(u)\notin D_{\overline{\epsilon}}$ as soon as $Y_{\lambda}\in D_{\epsilon}$.
Therefore
\begin{eqnarray*}
\lefteqn{W_0(\tau_{B}(Z_0)> s+T_{\lambda}-T_0;\tau_{B}(Z_{\lambda})\leq s)}\\
&\leq &\sum_{\epsilon\in E}W_0(Y_{\lambda}\in D_{\epsilon};\tau_{B}(Z_{\lambda})\leq s;Z_{\lambda}(u)\notin D_{\overline{\epsilon}})\;.
\end{eqnarray*}
Now, it is easily seen that $Y_{\lambda}$ is independent of $Z_{\lambda}$ conditionally to $X_1$. In addition, we have 
 $W_0(Y_{\lambda}\in D_{\epsilon}\vert\, X_1)=1/2^{d-1}$. Thus
\begin{eqnarray*}
\lefteqn{W_0(\tau_{B}(Z_0)> s+T_{\lambda}-T_0;\tau_{B}(Z_{\lambda})\leq s)}\\
& \leq &\frac{1}{2^{d-1}}\,\sum_{\epsilon\in E}W_0(\tau_{B}(Z_{\lambda})\leq s;Z_{\lambda}(u)\notin D_{\overline{\epsilon}})\\
& = &\frac{2^{d-1}-1}{2^{d-1}}\,W_0(\tau_{B}(Z_{\lambda})\leq s)\;.
\end{eqnarray*}
Combining this inequality with equation~\eqref{DECOMPPURMIEREVZZ} gives
\begin{equation*}\label{DECOMPPURMIEREVZA}
W_0(\tau_{B}(Z_{\lambda})\leq s)
\leq 2^{d-1}\,W_0(\tau_{B}(Z_0)\leq s+T_{\lambda}-T_0)
\end{equation*}
and the result follows by letting  $\lambda\to 0$ since
$\lim_{\lambda\to 0} T_{\lambda}=T_0$ almost surely (Lemma~\ref{unconv}).
\end{proof}

\begin{remark}\label{size} Note that the proof of Lemma~\ref{CONTRPOURBOULTPSSORT} does not involve the ``size'' of $B$. Hence the result holds for any open ball that is tangent to $\partial D$ at~$0$.
\end{remark}

\subsection{Application to nice sets}\label{nicenice}
In this section, we introduce the notion of \textit{nice} sets and solve the convergence problem for those sets.

\subsubsection{Convergence with variable sets}
For any set $U\subset\RR^d$ and any $\epsilon>0$, put 
$U_{\epsilon^+}=\{x\in\RR^d: d(x,U)\leq\epsilon\}$ 
and 
$U_{\epsilon^-}=\{x\in U: d(x,U^c)\geq \epsilon\}$.
If $(U_n)$ is a sequence of subsets of $\RR^d$, we will say that $(U_n)$ converges to $U$ and write $U_n\to U$ if for all $\epsilon>0$
there exists a $n_0$ such that 
$$n\geq n_0\Rightarrow U_{\epsilon^-}\subset U_n\subset U_{\epsilon^+}\;.$$
Let $D$ be the half-space $\{x\in\RR^d:x_1>0\}$ and let $B$ be an open ball tangent to $\partial D$ at $0$. Set $e_1=(1,0,\ldots,0)$.

\begin{proposition}\label{progress}
Let $U$ be an open co-regular set such that $B\subset U\subset D$ and let $(U_n)$ be a sequence of sets such that:
\begin{enumerate}
\item[(i)] For all $n$, $B\subset U_n\subset D$;
\item[(ii)] For all $R>0$, $U_n\cap B(0,R)\to U\cap B(0,R)$.
\end{enumerate}
Then, for all sequence $(\lambda_n)$ of positive numbers converging to $0$,
$$\W^{U_n}_{\lambda_ne_1,1}\Rightarrow \W^U_{0,1}\;.$$

\end{proposition}  
\begin{proof}
Set $x_n=\lambda_ne_1$. Since $\W^D_{0,1}(\tau_U>1)\geq\W^D_{0,1}(\tau_B>1)>0$, it suffices to prove that 
\begin{equation}\label{onfd}
\lim_{n\to\infty}\W_{x_n,1}^D(f;\tau_{U_n}>1)=\W^D_{0,1}(f;\tau_U>1)
\end{equation}
for all bounded continuous function $f:\C_{\infty}\to\RR$.

First suppose $U_n\to U$. Since each set $U_n$ contains the ball $B$, we have
\begin{equation}\label{dsfg}
\lim_{s\to 0}\limsup_n\W^D_{x_n,1}(\tau_{U_n}\leq s)=0
\end{equation}
by Lemma~\ref{CONTRPOURBOULTPSSORT}. 
As Shimura does in~\cite{Shi84}, Proof of Lemma 4.1, we
fix $s\in(0,1)$ and put $\tau^s_{U_n}=\inf\{t\geq s: X_t\notin U_n\}$ and 
$\tau^s_U=\inf\{t\geq s: X_t\notin U\}$. If $M$ is a bound for $\vert f\vert$, then
\begin{eqnarray*}
\lefteqn{\vert\W^D_{x_n,1}(f;\tau_{U_n}>1)-\W^D_{0,1}(f;\tau_U>1)\vert}\\
&\leq& \vert\W^D_{x_n,1}(f;\tau^s_{U_n}>1)-\W^D_{0,1}(f;\tau^s_U>1)\vert \\
&&+ M\left(\W^D_{x_n,1}(\tau_{U_n}\leq s)+\W^D_{0,1}(\tau_U\leq s)\right)\;.
\end{eqnarray*}
Hence~\eqref{onfd} will follow from~\eqref{dsfg} and the fact that $\W^D_{0,1}(\tau_U=0)=0$ if we prove that
\begin{equation}\label{jklm}
\lim_{n\to\infty}\W_{x_n,1}^D(f;\tau^s_{U_n}>1)=\W^D_{0,1}(f;\tau^s_U>1)
\end{equation}
for all $s\in(0,1)$. 
To do this, let us introduce also the random time $\tau^s_{\close{U}}=\inf\{t\geq s: X_t\notin \close{U}\}$
and set $\Omega=\{\tau^s_U=\tau^s_{\close{U}}\not=1\}$. 
We shall first prove that $\W^D_{0,1}(\Omega)=1$.
By the Markov property we have
$$\W^D_{0,1}(\tau_U^s=1)=\W^D_{0,1}\left(X_s\in U;\W^D_{X_s,1-s}(\tau_U=1-s)\right)=0$$
because $W_x(\tau_U=1-s)=0$ for all $x\in\RR^d$ (see \cite{Por78}, Theorem 4.7). We have also
\begin{eqnarray*}
\lefteqn{\W^D_{0,1}(\tau^s_U<\tau^s_{\close{U}})}\\
&=&\W^D_{0,1}(X_s\in U;\tau^s_U<\tau^s_{\close{U}})+\W^D_{0,1}(X_s\notin U;s<\tau^s_{\close{U}})\\
&=&\W^D_{0,1}\left(X_s\in U;\W^D_{X_s,1-s}(\tau_U<\tau_{\close{U}})\right)+\W^D_{0,1}\left(X_s\in\partial U;\W^D_{X_s,1-s}(\tau_{\close{U}}>0)\right)\\
&=& 0
\end{eqnarray*}
since $W_x(\tau_U<\tau_{\close{U}})=0$ for all $x\in \RR^d$ and $W_x(\tau_{\close{U}}>0)=0$ for all $x\in\partial U$ (remember that $U$ is co-regular).
Hence $\W^D_{0,1}(\Omega)=1$.

Now, since $U_n\to U$, it is easily seen that for all $w\in\Omega$ and every sequence $(w_n)\in \C_{\infty}$ such that $w_n\to w$,
$$
\indic{\{\tau^s_{U_n}>1\}}(w_n)\to\indic{\{\tau^s_U>1\}}(w)\;.
$$
Hence~\eqref{jklm} follows from the continuous mapping theorem (\cite{Bil99}, Theorem 2.7) and we have~\eqref{onfd} when $U_n\to U$.

Now we turn to the general case, that is we consider the local convergence hypothesis (ii) of Proposition~\ref{progress}. Fix $\epsilon>0$ and  choose $R>0$ such that $\W^D_{0,1}(\tau_{B(0,R)}>1)\geq 1-\epsilon$.
By the continuous mapping theorem, it is easily seen that
$\lim_{n\to\infty}\W^D_{x_n,1}(\tau_{B(0,R)}>1)=\W^D_{0,1}(\tau_{B(0,R)}>1)$.
Therefore
$\W^D_{x_n,1}(\tau_{B(0,R)}>1)\geq 1-2\epsilon$
for all large enough $n$.
Set $U'_n=U_n\cap B(0,R)$ and $U'=U\cap B(0,R)$.
Then 
\begin{eqnarray*}
\lefteqn{\vert\W^D_{x_n,1}(f;\tau_{U_n}>1)-\W^D_{0,1}(f;\tau_{U}>1)\vert}\\
&\leq& \vert\W^D_{x_n,1}(f;\tau_{U'_n}>1)-\W^D_{0,1}(f;\tau_{U'}>1)\vert+3M\epsilon
\end{eqnarray*}
where $M$ is a bound for $\vert f\vert$. By hypothesis $U'_n\to U'$, hence 
$$\limsup_n\vert\W^D_{x_n,1}(f;\tau_{U'_n}>1)-\W^D_{0,1}(f;\tau_{U'}>1)\vert=0$$
by the first step of this proof. Therefore
$$
\limsup_n\vert\W^D_{x_n,1}(f;\tau_{U_n}>1)-\W^D_{0,1}(f;\tau_{U}>1)\vert \leq 4M\epsilon
$$
and the desired result follows by letting $\epsilon\to 0$.
\end{proof}

\subsubsection{Nice sets}\label{nice}
Let $U$ be an open subset of $\RR^d$ and $x_0$ a boundary point of $U$.
We will say that $U$ is \textit{nice} at $x_0$ if there exist a neighborhood $V$ of $x_0$ and a number $r>0$ such that the following conditions are satisfied:
\begin{enumerate}
\item For all $x\in \partial U\cap V$ there exists a half-space $D_x\supset U$ such that:
\begin{itemize} 
\item $x\in\partial D_x$;
\item The ball $B_x\subset D_x$ with radius $r$ which is tangent to $\partial D_x$ at $x$ is contained in $U$;
\item The application $c$ which maps $x$ to the center $c(x)$ of the ball $B_x$ is continuous at $x_0$. 
\end{itemize}
\item For all $y\in U\cap V$ such that $d(y,\partial U)\leq r/2$, there exists a point $x=p(y)\in\partial U\cap V$ such that:
\begin{itemize}
\item $y\in (x,c(x)]$;
\item The mapping $y\mapsto p(y)$ is continuous.
\end{itemize}
\end{enumerate}
\begin{remark}One can check that regularity and convexity assumptions ensure the property of being a ``nice set''.
If the open set $U$ is convex and has a boundary of class $C^2$ in a neighborhood of  $x_0$ then the set $U$ is nice at $x_0$.
\end{remark}

Suppose  $U$ is nice at $x_0$. With the above notations, for any $x\in\partial U\cap V$, the point $x+c(x_0)-c(x)$ belongs to the boundary of the ball $B_{x_0}$; thus we can choose a planar rotation $R_x$ with center at $c(x_0)$ and such that
$R_x(x+c(x_0)-c(x))=x_0$. Note that the angle of $R_x$ tends to $0$ as $x\to x_0$, since $c(x)\to c(x_0)$.
Set $\phi_x(y)=R_x(y+c(x_0)-c(x))$ and $U_x=\phi_x(U)$.\\
Then it can be seen that
\begin{equation}\label{A}
B_{x_0}\subset U_x\subset D_{x_0}
\end{equation}
and
\begin{equation}\label{B}
U_x\cap B(0,R)\to U_{x_0}\cap B(0,R),\quad \mbox{as } x\to x_0\;,
\end{equation}
for all $R>0$.

\begin{theorem}\label{convnice}
Suppose  $U$ is co-regular and nice at $x_0$.\\
Then, as $x\in U\to x_0$, we have 
$\W^U_{x,1}\Rightarrow \W^U_{x_0,1}$.
\end{theorem}
\begin{proof}
For $y$ close to $x_0$, set $x=p(y)$.
Since $y$ belongs to $(x,c(x)]$, the point $q(y)=\phi_x(y)$ belongs to $(x_0,c(x_0)]$. Moreover, $q(y)$ tends to $x_0$ as $y\to x_0$.
Thus, from \eqref{A} and \eqref{B} together with Proposition~\ref{progress}, we obtain
$$\W^{U_x}_{q(y),1}\Rightarrow \W^U_{x_0,1},\quad\mbox{as }y\to x_0\;.$$
Now by the invariance properties of Brownian motion, we have
$$\W^U_{y,1}=\W^{U_x}_{q(y),1}\circ\phi_x\;.$$
Since $\phi_x$ tends to the identity mapping as $x\to x_0$, uniformly on compact subsets of $\RR^d$, it follows from the continuous mapping theorem that 
$$\W^U_{y,1}\Rightarrow \W^U_{x_0,1},\quad\mbox{as }y\to x_0\;.$$
\end{proof}

\section{Proof of Theorem~\ref{mainthm}}\label{ccones}
Let $d\geq 2$ and let $C\subset\RR^d$ be an  open  cone with vertex at $0$.
We will say that $C$ is a \textit{nice cone} if  it is nice (see~\ref{nice}) at any point of its boundary, excepting $0$.
For example, any two-dimensional convex cone is nice. In higher dimension, any circular cone or ellipsoidal
cone is nice.

We note two important facts about nice cones:
\begin{enumerate}
\item If $C$ is a nice cone, then it is a Lipschitz cone;
\item If $C$ is a nice cone, $\partial C$ is  a null set with respect to Lebesgue measure.
\end{enumerate}
The proof of the first one is elementary but quite tedious, so we omit it here. Note that the second fact is  a consequence of the first one.
 
The following lemma which is an immediate consequence of Theorem~\ref{convnice} will play an important role in the rest of this article.

\begin{lemma}\label{prolbord} Suppose $C$ is a nice cone.
Let $x_0\in\partial C\setminus\{0\}$ and $t_0>0$. As $(x,t)\to(x_0,t_0)$,
$\W^C_{x,t}\Rightarrow \W^C_{x_0/\sqrt{t_0},1}\circ K_{t_0}^{-1}$.
\end{lemma}
\begin{proof} By the scaling property of $\W^C$ (Remark~\ref{01Law}), we have
$\W^C_{x,t}=\W^C_{x/\sqrt{t},1}\circ K_t^{-1}$.
The result simply follows from Theorem~\ref{convnice} together with the continuous mapping theorem.
\end{proof}

\subsection{Convergence of the finite-dimensional distributions}

We will prove in this section that the finite-dimensional distributions of 
$\W^C_{x,1}$ converge weakly as $x\in C\to 0$. 
Recall that for any  $t\in(0,1]$ the law $\W^C_{x,1}(X_t\in dy)$ has the density $e_x(t,y)$ given by
$$e_x(t,y)=\frac{p^C(t,x,y)}{W_{x}(\tau_C>1)} W_y(\tau_C>1-t)\;.$$
By using an expansion of the heat kernel $p^C(t,x,y)$ of $C$  given by
Ba{\~n}uelos and Smits in~\cite{BS97}, we shall prove that $e_x(t,y)$ converges to a limit density $e(t,y)$, as  $x\in C\to 0$.
 
Before we recall their result, let us introduce some notations. 
Let $\O$ be the intersection of the cone $C$ with the unit sphere $\SS^{d-1}$ and assume that it is a regular set for the Dirichlet problem with respect to the Laplace-Beltrami operator $L$ on $\SS^{d-1}$.
Then there exists a complete set of orthonormal eigenfunctions $m_j$ with corresponding eigenvalues 
$0<\lambda_1<\lambda_2\leq \lambda_3\leq\cdots$ satisfying
$$
\begin{cases}
L m_j=-\lambda_j m_j & \mbox{on }\O\;;\\
m_j=0 &\mbox{on }\partial \O\;.
\end{cases}
$$ 
Set $\alpha_j=\sqrt{\lambda_j+(\frac{d}{2}-1)^2}$.
We will use the following facts that are proved in~\cite{BS97}~:
\begin{itemize}
\item there exist two constants $0<c_1<c_2$ such that 
\begin{equation}\label{croissalp}
\forall j\geq 1,\quad c_1j^{\frac{1}{d-1}}\leq \alpha_j \leq c_2 j^{\frac{1}{d-1}}\;;
\end{equation}
\item there exists a constant $c$ such that 
\begin{equation}\label{majfoncprop}
\forall j\geq 1,\quad \Vert m_j \Vert_{\infty}\leq c\alpha_j^{\frac{d-1}{2}}\;;
\end{equation}
\item if $C$ is a Lipschitz cone, then there exists a constant $c'$ such that
\begin{equation}\label{compfoncprop}
\forall j\geq 1, \forall \eta\in \O,\quad m_j^2(\eta)\leq \frac{c'm_1^2(\eta)}{I_{\alpha_j}(1)}\;,
\end{equation}

\end{itemize}
where $I_{\nu}$ is the modified Bessel function of order $\nu$~:
\begin{eqnarray}\label{bessel}
I_{\nu}(x)&=&\frac{2(\frac{x}{2})^{\nu}}{\sqrt{\pi}\Gamma(\nu+\frac{1}{2})}\int_{0}^{\frac{\pi}{2}}(\sin t)^{2\nu}\cosh(x\cos t)\,dt\\
&=&\sum_{m=0}^{\infty}\frac{(\frac{x}{2})^{\nu+2m}}{m!\Gamma(\nu+m+1)}\nonumber\;.
\end{eqnarray}
Then we have the following lemma~:
\begin{lemma}[\cite{BS97}, Lemma 1]
Write $x=\rho\theta$, $y=r\eta$, $\rho$, $r>0$, $\theta$, $\eta\in \O$. We have
$$p^C(t,x,y)=\frac{e^{-\frac{(r^2+\rho^2)}{2t}}}{t(\rho r)^{\frac{d}{2}-1}}
\sum_{j=1}^{\infty} I_{\alpha_j}\left(\frac{\rho r}{t}\right)m_j(\theta)m_j(\eta)\;,$$
where the convergence is uniform for $(t,x,y)\in [T,\infty)\times\{x\in C:\Vert x\Vert\leq R\}\times C$, for any positive constants $T$ and $R$.
\end{lemma}
Together with the expression of $I_{\alpha_j}$, this suggests that $p^C(t,x,y)$ is equivalent at $x=0$ to the product  $g(x)h(t,y)$ where 
$$g(x)=\rho^{\alpha_1-(\frac{d}{2}-1)}m_1(\theta)
\quad\mbox{ and }\quad
h(t,y)=\frac{r^{\alpha_1-(\frac{d}{2}-1)}e^{-\frac{r^2}{2t}}}{2^{\alpha_1}\Gamma(\alpha_1+1)t^{\alpha_1+1}}m_1(\eta)\;.$$
In fact, we have the following~:
\begin{lemma}\label{equibrotue}
For $x=\rho\theta$, $y=r\eta$, $\rho$, $r>0$, $\theta$, $\eta\in \O$, we have
$$\lim_{\rho\to 0}\frac{p^C(t,x,y)}{g(x)}=
h(t,y)\;,$$
uniformly in $(t,r,\theta,\eta)\in [T,\infty)\times [0,R]\times \O\times \O$, for any positive constants $T$ and $R$.
\end{lemma}
\begin{proof} 
Throughout this proof, the letter $\kappa$ will denote some positive constant whose value may change from line to line.\\
Set $M=\frac{\rho r}{t}$. We have
$$\frac{p^C(t,x,y)}{g(x)h(t,y)}=2^{\alpha_1}\Gamma(\alpha_1+1)e^{-\frac{\rho^2}{2t}}\sum_{j=1}^{\infty}\frac{I_{\alpha_j}(M)}{M^{\alpha_1}}
\frac{m_j(\theta)}{m_1(\theta)}\frac{m_j(\eta)}{m_1(\eta)}\;.$$
Using relation~(\ref{compfoncprop}), we get
$$\left\vert \frac{I_{\alpha_j}(M)}{M^{\alpha_1}}\frac{m_j(\theta)}{m_1(\theta)}\frac{m_j(\eta)}{m_1(\eta)}\right\vert
\leq\frac{\kappa}{M^{\alpha_1}}\frac{I_{\alpha_j}(M)}{I_{\alpha_j}(1)}\;.$$
Now, using the integral expression for $I_{\alpha_j}$, we see that
$$\frac{I_{\alpha_j}(M)}{I_{\alpha_j}(1)}\leq M^{\alpha_j}\cosh{M}\;.$$
Hence
$$\left\vert \frac{I_{\alpha_j}(M)}{M^{\alpha_1}}\frac{m_j(\theta)}{m_1(\theta)}\frac{m_j(\eta)}{m_1(\eta)}\right\vert
\leq\kappa M^{\alpha_j-\alpha_1}\cosh{M}\;.$$
From relation~(\ref{croissalp}), it is easily seen that the series  $\sum_j M^{\alpha_j-\alpha_1}\cosh{M}$ is uniformly convergent on $[0,1-\epsilon]$. So, the series
$$\sum_j \frac{I_{\alpha_j}(M)}{M^{\alpha_1}}\frac{m_j(\theta)}{m_1(\theta)}\frac{m_j(\eta)}{m_1(\eta)}$$
is uniformly convergent for $(M,\theta,\eta)\in [0,1-\epsilon]\times \O\times \O$.
Therefore we can take the limit term by term to obtain
$$\lim_{M\to 0}\sum_{j=1}^{\infty} \frac{I_{\alpha_j}(M)}{M^{\alpha_1}}\frac{m_j(\theta)}{m_1(\theta)}\frac{m_j(\eta)}{m_1(\eta)}=\frac{1}{2^{\alpha_1}\Gamma(\alpha_1+1)}\;,$$
where the convergence is uniform for $(\theta,\eta)\in \O\times \O$. 
\end{proof}

\begin{lemma}\label{domination}
The  function of  $y$
$$\sup_{\Vert x\Vert\leq \frac{1}{2}}\left \vert \frac{p^C(1,x,y)}{g(x)}\right\vert$$
is integrable on $\RR^d$.
\end{lemma}
\begin{proof} 
Using relations~(\ref{majfoncprop}) and~(\ref{compfoncprop}), we get
\begin{eqnarray*}
\left \vert \frac{p^C(1,x,y)}{g(x)}\right\vert & \leq & \frac{e^{-\frac{(\rho^2+r^2)}{2}}}{r^{\frac{d}{2}-1}\rho^{\alpha_1}}\sum_{j=1}^{\infty}
I_{\alpha_j}(\rho r)\left\vert \frac{m_j(\theta)}{m_1(\theta)} m_j(\eta)\right\vert\\
 & \leq & \frac{e^{-\frac{r^2}{2}}}{r^{\frac{d}{2}-1}\rho^{\alpha_1}}\sum_{j=1}^{\infty}
\frac{I_{\alpha_j}(\rho r)}{I_{\alpha_j}(1)^{\frac{1}{2}}}\alpha_j^{\frac{d-1}{2}}\;.
\end{eqnarray*}
Set $\omega_{\alpha_j}=\int_{0}^{\frac{\pi}{2}}(\sin t)^{2\alpha_j}\,dt$. Using the integral expression for $I_{\alpha_j}$, we find that
$$\frac{I_{\alpha_j}(\rho r)}{I_{\alpha_j}(1)^{\frac{1}{2}}}\leq
\kappa \cosh(\rho r)\left(\frac{\rho r}{\sqrt{2}}\right)^{\alpha_j}\frac{\sqrt{\omega_{\alpha_j}}}{\Gamma(\alpha_j+\frac{1}{2})^{\frac{1}{2}}}\;.$$
Since $\int_{0}^{\frac{\pi}{2}}(\sin t)^{2n}\,dt\sim cn^{-\frac{1}{2}}$ as $n\to\infty$, we have $\omega_{\alpha_j}\sim c\alpha_j^{-\frac{1}{2}}$ as $j\to\infty$. From Stirling's Formula we also get $\Gamma(\alpha_j+\frac{1}{2})\geq c\alpha_j^{\alpha_j}e^{-\alpha_j}$. Thus,
$$\frac{I_{\alpha_j}(\rho r)}{I_{\alpha_j}(1)^{\frac{1}{2}}}\leq
\kappa \cosh(\rho r)\left(\frac{\sqrt{e}\rho r}{\sqrt{2}}\right)^{\alpha_j}\frac{\alpha_j^{-\frac{1}{4}}}{\alpha_j^{\alpha_j/2}}\;.$$
Therefore,
\begin{eqnarray}\label{anniv}
\left \vert \frac{p^C(1,x,y)}{g(x)}\right\vert & \leq & \kappa \frac{e^{-\frac{r^2}{2}}}{r^{\frac{d}{2}-1}\rho^{\alpha_1}}\cosh(\rho r)\sum_{j=1}^{\infty}
 \left(\frac{\sqrt{e}\rho r}{\sqrt{2}}\right)^{\alpha_j}\frac{\alpha_j^{\frac{2d-3}{4}}}{\alpha_j^{\alpha_j/2}}\;.
\end{eqnarray}
Since $\alpha_j\geq\alpha_1$, the right-hand side of~\eqref{anniv} is increasing with $\rho$, so
\begin{eqnarray*}
\sup_{\rho\leq \frac{1}{2}}\left \vert \frac{p^C(1,x,y)}{g(x)}\right\vert & \leq & \kappa \frac{e^{-\frac{r^2}{2}}}{r^{\frac{d}{2}-1}}\cosh\left(\frac{r}{2}\right)\sum_{j=1}^{\infty}
 \left(\frac{\sqrt{e} r}{2\sqrt{2}}\right)^{\alpha_j}\frac{\alpha_j^{\frac{2d-3}{4}}}{\alpha_j^{\alpha_j/2}}=:f(r)\;.
\end{eqnarray*}
Because  $\alpha_j>\left(\frac{d}{2}-1\right)$, the function $f$ is integrable on any compact subset of $[0,+\infty)$. 
We shall now find an upper bound for the sum that appears in the definition of $f$ for large values of $r$.
Let $M\geq 1$. For $2n\leq \alpha_j\leq 2n+1$, we have
$$M^{\alpha_j}\frac{\alpha_j^{\frac{2d-3}{4}}}{\alpha_j^{\alpha_j/2}}\leq M^{2n+1}\frac{(2n+1)^{\frac{2d-3}{4}}}{(2n)^n}=M (M^2/2)^n 
\frac{(2n+1)^{\frac{2d-3}{4}}}{n^n}\;.$$
Since $\alpha_j>c_1 j^{\frac{1}{d-1}}$, the number of indices $j$ for which  $\alpha_j\leq 2n+1$ is bounded by $\left(\frac{2n+1}{c_1}\right)^{d-1}$. Thus, there exists $K=K(d)>0$ such that
$$\sum_{j=1}^{\infty}M^{\alpha_j}\frac{\alpha_j^{\frac{2d-3}{4}}}{\alpha_j^{\alpha_j/2}}\leq
M\sum_{n=1}^{\infty}(M^2/2)^n\frac{n^K}{n^n}\leq P(M)e^{M^2/2}\;,$$
where $P$ is a polynomial. Applying this result with $M=\frac{\sqrt{e}r}{2\sqrt{2}}$ and $r\geq 2\sqrt{2/e}$ gives
\begin{eqnarray*}
f(r)& \leq & \kappa \frac{e^{-(1-e/8)\frac{r^2}{2}}}{r^{\frac{d}{2}-1}}\widetilde{P}(r) \cosh\left(\frac{r}{2}\right)\;,
\end{eqnarray*}
where $\widetilde P$ is a polynomial whose coefficients depend only on $d$.
This is sufficient to conclude the proof of Lemma~\ref{domination}.
\end{proof}

Recall that
$$
e_x(t,y)=\frac{p^C(t,x,y)}{\int p^C(1,x,z)\,dz}W_y(\tau_C>1-t)\;.
$$
Lemma~\ref{equibrotue} suggests that the limit as $x\to 0$ is the function
$$
e(t,y)=\frac{h(t,y)}{\int_Ch(1,z)\,dz}W_y(\tau_C>1-t)\;.
$$
By integrating in polar coordinates, it is easily seen that
$$\int_C h(1,z)\,dz=2^{-\frac{\alpha_1}{2}+\frac{d-2}{4}}\frac{\Gamma(\frac{\alpha_1}{2}+\frac{d+2}{4})}{\Gamma(\alpha_1+1)}\int_{\O}m_1(\eta)\,\sigma(d\eta)\;,$$
where $\sigma$ is Lebesgue measure on the unit sphere $\SS^{d-1}$.
Put $$c^{-1}=2^{\frac{\alpha_1}{2}+\frac{d-2}{4}}\Gamma\left(\frac{\alpha_1}{2}+\frac{d+2}{4}\right)\int_{\O} m_1(\eta)\,\sigma(d\eta)\;.$$
Then, for $y=r\eta$, we get
\begin{equation}\label{densite}
e(t,y)=
ct^{-\alpha_1-1} r^{\alpha_1-(\frac{d}{2}-1)}e^{-r^2/2t} m_1(\eta)\,W_y(\tau_C>1-t)\;.
\end{equation}
We have the following result:
\begin{lemma}\label{denslimpretran}For any $t\in (0,1]$, the function $e(t,y)$ is a probability density and
$e_x(t,y)\to e(t,y)$ as $\Vert x\Vert\to 0$.
\end{lemma}
\begin{proof}
We shall first prove that the the family $\{e_x(t,y):\Vert x\Vert\leq 1\}$ is equi-integrable, that is
\begin{equation}\label{equiint}
\lim_{R\to\infty}\sup_{\Vert x\Vert\leq 1}\W^C_{x,1}(\Vert X_t\Vert> R)=0\;.
\end{equation}

Let $x\in C$ with $\Vert x\Vert\leq 1$ be given, and let $R>2$. We denote by $\rho=\tau_{B(0,2)}$ the first exit time from the ball $B(0,2)$. A continuous path started at $x$ that is outside $B(0,R)$ at time $t$ must have left $B(0,2)$ before that time, so
\begin{eqnarray*}
\lefteqn{\W^C_{x,1}(\Vert X_t\Vert >R)}\\
&=&\W^C_{x,1}(\rho<t;\W^C_{X_{\rho},1-\rho}(\Vert X_{t-s}\Vert>R)_{\vert s=\rho})\\
&\leq&\sup\left\{\W^C_{y,1-s}(\Vert X_{t-s}\Vert>R): y\in C,\Vert y\Vert=2\mbox{ and }s\in[0,t]\right\}\;.
\end{eqnarray*}
Suppose  the last expression does not tend to $0$ as $R\to \infty$; then there exist a sequence
$(y_n)\in C$ with $\Vert y_n\Vert=2$ and a sequence $(s_n)\in[0,t]$ such that 
\begin{equation}\label{contradict}
\liminf_{n\to\infty}\W^C_{y_n,1-s_n}(\Vert X_{t-s_n}\Vert >n)>0\;.
\end{equation}
Without loss of generality, we may suppose that $(y_n)$ converges to a point $y\in \close{C}$ with $\Vert y\Vert=2$, and that $(s_n)$ converges to $s\in[0,t]$. But Lemma~\ref{prolbord} (or Proposition~\ref{contthm} if $y\in C$) then implies that $\left(\W^C_{y_n,1-s_n}(X_{t-s_n}\in dy)\right)$ is a convergent sequence of probability measures : this contradicts~\eqref{contradict}. Thus, equation~\eqref{equiint} is proven.

It follows from Lemmas~\ref{equibrotue},~\ref{domination} and the dominated convergence theorem that 
\begin{equation}\label{equitepssor}
\lim_{\Vert x\Vert\to 0}\frac{\int p^C(1,x,z)\,dz}{g(x)}=\int h(1,z)\,dz\;.
\end{equation}
Since 
$$e_x(t,y)=\frac{p^C(t,x,y)}{g(x)}\frac{g(x)}{\int p^C(1,x,z)\,dz}W_y(\tau_C>1-t)\;,$$
we deduce from Lemma~\ref{equibrotue} and relation~\eqref{equitepssor} that $e_x(t,y)\to e(t,y)$, as $\Vert x\Vert\to 0$, uniformly on
$\{y\in C: \Vert y\Vert\leq R\}$, for any positive constant $R$.
Thus, it follows from~\eqref{equiint} and the integrability of $e(t,y)$ that
$$\limsup_{\Vert x\Vert\to 0}\int\left\vert e_x(t,y)-e(t,y)\right\vert\,dy=0\;.$$
This proves that the function $y\mapsto e(t,y)$ is a probability density. 
\end{proof}

\begin{proposition}\label{COMARUNDIMBM}
The finite-dimensional distributions of $\W^C_{x,1}$ converge weakly as $x\in C$ tends to $0$. Moreover, the limit distribution of the first transition law $\W^C_{x,1}(X_t\in dy)$, $t\in (0,1]$, has the density $e(t,y)$ given by equation~\eqref{densite}.
\end{proposition}
\begin{proof}
It follows from Lemma~\ref{denslimpretran} that the laws $\W^C_{x,1}(X_t\in dy)=e_x(t,y)\,dy$ converge weakly to
$e(t,y)\,dy$ as $x\in C$ tends to~$0$. The weak convergence of the finite-dimensional distributions then follows from Proposition~\ref{critconvmarg} since $\partial C$ has Lebesgue measure $0$.
\end{proof}

\subsection{Tightness}
For any $T>0$, the space $\C_T$ of continuous paths $w:[0,T]\to\RR^d$ is endowed with the topology generated by the supremum metric and the corresponding Borel $\sigma$-algebra.

\begin{proposition}\label{tenssuite}
For any sequence $(x_n)$ of points of~$C$ converging to~$0$ and for any $T>0$, the sequence of probability measures $(\W^C_{x_n,1})$ is tight in 
$\C_{T}$.
\end{proposition}
\begin{proof}
Our proof is a modification of Shimura's one for the two-dimensional case~(\cite{Shi85}, Theorem~2). 
Since the  arguments do not depend on the value of $T$, we will only consider the case $T=1$.
 It suffices to prove that, for all $\epsilon>0$,
$$\lim_{\delta\to 0}\limsup_{n\to\infty}\W^C_{x_n,1}(\chi(\delta,0,1)>\epsilon)=0\;,$$
where 
$\chi(\delta,a,b)(w)=\sup\{\Vert w(s)-w(t)\Vert:\vert s-t\vert\leq\delta,s,t\in[a,b]\}$
is the modulus of continuity of order $\delta$ of $w$ on $[a,b]$ 
(see Billingsley~\cite{Bil99}, Theorem 7.3).

Fix $\epsilon>0$ and set $s=1/2$. Since $\chi(\delta,\cdot,\cdot)$ is subadditive when considered as a function on the set of intervals, we have
$$\W^C_{x_n,1}(\chi(\delta,0,1)>4\epsilon)\leq \underbrace{\W^C_{x_n,1}\left(\chi(\delta,0,s)>3\epsilon\right)}_{A_n(\delta)} +\underbrace{\W^C_{x_n,1}\left(\chi(\delta,s,1)>\epsilon\right)}_{B_n(\delta)}\;.$$ 
Let us start with  $B_n(\delta)$. It follows from Proposition~\ref{COMARUNDIMBM} that 
$$\lim_{r\to 0,R\to\infty}\liminf_{n\to\infty}\W^C_{x_n,1}\left(r\leq \Vert X_{s}\Vert\leq R\right)=1\;.$$
Hence we can fix $\alpha>0$ and choose $0<r<R$ such that 
$$\inf_n\W^C_{x_n,1}\left(r\leq \Vert X_{s}\Vert\leq R\right)\geq 1-\alpha\;.$$
We then have
$$B_n(\delta)\leq \W^C_{x_n,1}\left(r\leq \Vert X_{s}\Vert\leq R;\chi(\delta,s,1)>\epsilon\right)+\alpha\;,$$
So, by the Markov property,
\begin{eqnarray*}
B_n(\delta)& \leq &\W^C_{x_n,1}\left(r\leq \Vert X_{s}\Vert\leq R;\W^C_{X_s,1-s}(\chi(\delta,0,1-s)>\epsilon)\right)+\alpha\\
   & \leq & \underbrace{\sup\left\{\W^C_{y,s}\left(\chi(\delta,0,s)>\epsilon)\right):y\in C\mbox{ and }r\leq \Vert y\Vert\leq R\right\}}_{D(\delta)}+\alpha\;.
\end{eqnarray*}
Now, if $D(\delta)$ did not tend to $0$ as  $\delta$ goes to $0$, then we could find a sequence $(\delta_n)$ converging to $0$ and a sequence $(y_n)$ of points of $C$ converging to a point $y\in \overline{C}\setminus\{0\}$ such that
$\liminf_{n}\W^C_{y_n,s}\left(\chi(\delta_n,0,s)>\epsilon\right)>0$,
which would contradict the weak convergence of the sequence of probability measures
$\left(\W^C_{y_n,s}\right)$ (Lemma~\ref{prolbord} or Proposition~\ref{contthm} if $y\in C$).
This proves that $\lim_{\delta\to 0}\limsup_n B_n(\delta)\leq \alpha$, and letting $\alpha\to 0$ then gives 
$$\lim_{\delta\to 0}\limsup_{n\to\infty} B_n(\delta)=0\;.$$

We now turn to $A_n(\delta)$.
Let $\rho=\tau_{B(0,\epsilon)}$ be the exit time from the ball $B(0,\epsilon)$ with center at $0$ and radius  $\epsilon$. 
Since the modulus of continuity of a path $w$ is less than $2\epsilon$ as long as it has not left the ball $B(0,\epsilon)$, we have
\begin{eqnarray*}
A_n(\delta) & \leq & \W^C_{x_n,1}\left(\rho<s;\chi(\delta,\rho,s)>\epsilon\right)\\
&\leq & \W^C_{x_n,1}\left(\rho<s;\W^C_{X(\rho),1-\rho}(\chi(\delta,0,1)>\epsilon)\right)\;.
\end{eqnarray*}
Hence
$$
\limsup_{n\to\infty} A_n(\delta) 
 \leq  \sup\{\W^C_{y,t}(\chi(\delta,0,1)>\epsilon):y\in C, \Vert y\Vert=\epsilon\mbox{ and } t\in [s,1]\}\;.
$$
In the same way as above, we then get
$\limsup_{n} A_n(\delta)=0$,
which is sufficient to prove Proposition~\ref{tenssuite}.
\end{proof}

Together with Proposition~\ref{COMARUNDIMBM}, Proposition~\ref{tenssuite} proves that $\W^C_{x,1}$ converges weakly on every $\C_T$, $T>0$, as $x\in C$ tends to $0$. This is equivalent to weak convergence on $\C_{\infty}$; thus Theorem~\ref{mainthm} is proven.

The limit law will be denoted by $\W^C_{0,1}$ and called the law of $C$-Brownian meander.
In view of Theorem~\ref{mainthm}, we shall interpret
the $C$-Brownian meander as a Brownian motion conditioned to stay in $C$ for a unit of time.

\subsection{Some properties of the $C$-Brownian meander}
Since $\W_{0,1}^C(X_t\in dy)$ has a probability density $e(t,y)$ for each $t\in (0,1)$, it follows from Proposition~\ref{prolongement} that 
$W_{0,1}^C(\tau_C>1)=1$ and that the $C$-Brownian meander satisfies the following Markov property:
For all $t> 0$, $A\in \F_{t^+}$ and $B\in\F$,
\begin{equation*}
\W^C_{0,1}(A;\theta_t^{-1}B)=\W^C_{0,1}\left(A;\W^C_{X_t,1-t}(B)\right)\;.
\end{equation*} 

The $C$-Brownian meander starts from the vertex of the cone $C$ and stays in it for a unit of time. The law of its exit time from $C$ after time $1$ is given in the next proposition.

\begin{proposition}\label{resteencore} For any $t>1$, we have
$$\W^C_{0,1}(\tau_C>t)=t^{-\frac{\alpha_1}{2}+\frac{d-2}{4}}\;.$$
\end{proposition}
\begin{proof}
By the Markov property, we have
$$
\W^C_{0,1}(\tau_C>t)=\W^C_{0,1}\Big(W_{X_1}(\tau_C>t-1)\Big)
=\int_C e(1,y)W_y(\tau_C>t-1)\,dy\;.
$$
With the change of variables $y=\sqrt{t}u$, the last integral becomes
$$\int_C e(1,\sqrt{t}u)W_{\sqrt{t}u}(\tau_C>t-1)\,t^{\frac{d}{2}}\,du\;.$$
But  
$W_{\sqrt{t}u}(\tau_C>t-1)=W_u(\tau_C>1-1/t)$
by the scaling property of Brownian motion, and from relation~\eqref{densite} p.~\pageref{densite} it is easily seen that
$$e(1,\sqrt{t}u)W_u(\tau_C>1-1/t)\,t^{\frac{d}{2}}=t^{-\frac{\alpha_1}{2}+\frac{d-2}{4}}e(1/t,u)\;.$$
The expected result follows from the fact that $e(1/t,u)$ is a probability density.
\end{proof}

\section*{Acknowledgments}
This paper is related to the \textit{Th\`ese de doctorat} of the author. We would like to thank our advisors, Emmanuel Lesigne and Marc Peign\'e for their constant help in preparing the thesis and the present paper. We would also like to thank the referee for helping in many ways to improve the original manuscript.



\begin{thebibliography}{10}
\bibitem{BS97}
Ba{\~n}uelos, R. and Smits, R.G. (1997).
\newblock Brownian motion in cones.
\newblock {\em Probab. Theory Related Fields} {\bf 108} (3) : 299--319.

\bibitem{Bil99}
Billingsley, P. (1999).
\newblock {\em Convergence of Probability Measures.} Second edition.
\newblock Wiley, New York.

\bibitem{Bur86}
Burdzy, K. (1986).
\newblock Brownian excursions from hyperplanes and smooth surfaces.
\newblock {\em Trans. Amer. Math. Soc.} {\bf 295} (1) : 35--57.


\bibitem{DIM77}
Durrett, R.~T., Iglehart, D.~L. and Miller, D.~R. (1977).
\newblock Weak convergence to Brownian meander and Brownian excursion.
\newblock {\em Ann. Probab.} {\bf 5} (1) : 117--129.


\bibitem{Por78}
Port, S.~C. and Stone, C.~J. (1978).
\newblock {\em Brownian Motion and Classical Potential Theory}.
\newblock Academic Press, New York.


\bibitem{Shi84}
Shimura, M. (1984).
\newblock A limit theorem for two-dimensional conditioned random walk.
\newblock {\em Nagoya Math. J.} {\bf 95} : 105--116.

\bibitem{Shi85}
Shimura, M. (1985).
\newblock Excursions in a cone for two-dimensional Brownian motion.
\newblock {\em J. Math. Kyoto Univ.} {\bf 25} (3) : 433--443.

\end{thebibliography}
\end{document}